\documentclass[journal]{IEEEtran}

\usepackage[cmex10]{amsmath}
\usepackage{amsthm, amsfonts, amssymb, graphicx, cite, texdraw}

\newtheorem{lemma}{Lemma}
\newtheorem{theorem}{Theorem}

\begin{document}

\title{Nonlinear Dual-Mode Control of Variable-Speed Wind Turbines with Doubly Fed Induction Generators}

\author{Choon~Yik~Tang,~\IEEEmembership{Member,~IEEE,}
        Yi~Guo,~\IEEEmembership{Student Member,~IEEE,}
        and~John~N.~Jiang,~\IEEEmembership{Senior~Member,~IEEE}
\thanks{The authors are with the School of Electrical and Computer Engineering, University of Oklahoma, Norman, OK 73019 USA (e-mail: \{cytang,yi.guo,jnjiang\}@ou.edu).}
\thanks{This work was supported by the National Science Foundation under grants ECCS-0926038 and ECCS-0955265.}}

\maketitle

\begin{abstract}
This paper presents a feedback/feedforward nonlinear controller for variable-speed wind turbines with doubly fed induction generators. By appropriately adjusting the rotor voltages and the blade pitch angle, the controller simultaneously enables: (a) control of the active power in both the {\em maximum power tracking} and {\em power regulation} modes, (b) seamless switching between the two modes, and (c) control of the reactive power so that a desirable power factor is maintained. Unlike many existing designs, the controller is developed based on original, nonlinear, electromechanically-coupled models of wind turbines, without attempting approximate linearization. Its development consists of three steps: (i) employ feedback linearization to exactly cancel some of the nonlinearities and perform arbitrary pole placement, (ii) design a speed controller that makes the rotor angular velocity track a desired reference whenever possible, and (iii) introduce a Lyapunov-like function and present a gradient-based approach for minimizing this function. The effectiveness of the controller is demonstrated through simulation of a wind turbine operating under several scenarios.
\end{abstract}


\section{Introduction}\label{sec:intro}

\IEEEPARstart{W}{ind} power is gaining ever-increasing attention in recent years as a clean, safe, and renewable energy source. With the fast growth of wind generation in power systems, wind power is becoming a significant portion of the generation portfolio in the United States as well as many countries in Europe and Asia \cite{Eriksen05}. Indeed, wind power penetration is planned to surpass 20\% of the United States' total energy production by 2030---a figure that is way beyond the current level of less than 5\% \cite{USDOE08}. Hence, to realize this vision, it is necessary to develop large-scale wind farms that effectively produce electric power from wind, and integrate them with the power systems.

The integration of large-scale wind farms into a power system, however, changes the fundamental principle of its operation, which is to maintain reliability by balancing load variation with ``controllable'' generation resources. When a portion of these resources comes from ``uncontrollable'' wind generation, that portion of the resources can hardly be guaranteed due to the intermittency of wind. As a result, the power system may fail to achieve the required balance. When the level of wind power penetration is small, this issue may be safely neglected. However, with the anticipated increase in penetration, the issue becomes critical for power system reliability. As a case in point, half of the European grids experienced a severe difficulty in 2006 because several large wind farms, operating in the {\em maximum power tracking} (MPT) mode, produced excessive power that destabilized the grids. This event took place even though the wind penetration level, at that time, was low at only 7\% \cite{UCTE06}.

Such an experience suggests that it is not advisable---and may be even disastrous---for a large-scale wind farm connected to a grid to always operate in the MPT mode, making wind turbines in the farm harvest as much wind energy as they possibly could, following the ``let it be when the wind blows'' philosophy of operation. Instead, it is highly desirable that the wind farm can also operate in the {\em power regulation} (PR) mode, whereby its total power output from the wind turbines is closely regulated at some desired setpoint, despite the fluctuating wind conditions.

The ability to operate in the PR mode in addition to the MPT mode, as well as the ability to {\em seamlessly} switch between the two, offers many important advantages: not only does the PR mode provide a cushion to absorb the impact of wind fluctuations on total power output through power regulation, it also enables a power system to effectively respond to changes in reliability conditions and economic signals. For instance, when a sudden drop in load occurs, the power system may ask the wind farm to switch from the MPT to the PR mode and generate less power, rather than rely on expensive down-regulation generation. As another example, the PR mode, when properly designed, allows the power output of a wind farm to smoothly and accurately follow system dispatch requests, thus reducing its reliance on ancillary services such as reliability reserves.

To enable large-scale wind farms to operate well in these two modes and switch seamlessly between them, numerous challenges must be overcome. This paper is devoted to addressing a subset of these challenges, by presenting an integrated framework for controlling the rotor voltages and the blade pitch angle of variable-speed wind turbines with doubly fed induction generators (DFIGs). The paper presents a feedback/feedforward nonlinear controller developed based on original, nonlinear, and electromechanically-coupled models of wind turbines, without attempting approximate linearization. The controller simultaneously enables: (a) control of the active power in both the MPT and PR modes, (b) seamless switching between the two modes, and (c) control of the reactive power so that a desirable power factor is ensured. Its development consists of three steps. First, we show that, although dynamics of a wind turbine are highly nonlinear and electromechanically coupled, they offer a structure, which makes the electrical part feedback linearizable, so that arbitrary pole placement can be carried out. Second, we show that because the electrical dynamics can be made very fast, it is possible to perform model order reduction, so that only the first-order mechanical dynamics remain to be considered. For this reduced first-order model, a speed controller is designed, which enables the rotor angular velocity to track a desired reference whenever possible. Finally, we introduce a Lyapunov-like function that measures the difference between the actual and desired powers and present a gradient-based approach for minimizing this function. The effectiveness of the controller is demonstrated through simulation of a wind turbine operating under a changing wind speed, changing desired power outputs, modeling errors, and noisy measurements.

To date, a significant amount of research has been performed on the control of variable-speed wind turbines \cite{Prats00,Iyasere08,Johnson04b,Johnson06,Calderaro07,Galdi08,Chedid00,Beltran08,Malinga03,Hand99,ZhangL08,GengH09,Muljadi01,Senjyu06,Stol01,Wright03,Wright03b,ZhangJZ08,Janssens07,Tarnowski07,Ko07,Brekken03,Marinescu04,LiDD04,ZhiDW07,WuF08, Alegria04,Pena96,Hopfensperger00,Monroy08,Peresada04,Rabelo01,LiH06,Hansen04}. The existing publications, however, are substantially different from our work in the following aspects:
\begin{enumerate}
\renewcommand{\theenumi}{(\roman{enumi})}\itemsep-\parsep
\renewcommand{\labelenumi}{\theenumi}
\item The mechanical and electrical parts of the wind turbines are considered separately in most of the current literature: \cite{Iyasere08,Malinga03,Hand99,Chedid00,Prats00,Johnson04b,Johnson06,Beltran08,Galdi08,GengH09,Muljadi01,Senjyu06,Stol01,Wright03,Wright03b,ZhangL08,ZhangJZ08,Calderaro07} considered only the mechanical part, while \cite{Ko07,Janssens07,ZhiDW07,Marinescu04,Tarnowski07,WuF08,Pena96,Monroy08,Peresada04,LiDD04,Brekken03,Alegria04,Hopfensperger00,Rabelo01,LiH06} considered only the electrical part, focusing mostly on the DFIGs. In contrast, in this paper we consider {\em both} the mechanical and electrical parts. Although \cite{Hansen04} also considered both these parts, its controller was designed to maximize wind energy conversion, as opposed to achieving power regulation (i.e., only operate in the MPT mode). In comparison, our controller can operate in {\em both} the MPT and PR modes as well as seamlessly switch between the two.
\item For those references \cite{Iyasere08,Malinga03,Hand99,Chedid00,Prats00,Johnson04b,Johnson06,Beltran08,Galdi08,GengH09,Muljadi01,Senjyu06,Stol01,Wright03,Wright03b,ZhangL08,ZhangJZ08,Calderaro07} considering only the mechanical part, control of the active power has been the main focus: \cite{Iyasere08,Prats00} maximized the wind power capture at low to medium wind speeds by adjusting both the generator torque and the blade pitch angle, \cite{Johnson04b,Johnson06,Galdi08,Calderaro07} did the same by adjusting only the generator torque, and \cite{Hand99,GengH09,Muljadi01,Senjyu06,Stol01,Wright03,Wright03b,ZhangL08,ZhangJZ08} aimed to maintain the rated rotor speed and limit the power production at high wind speeds by controlling the blade pitch angle. Although the results are interesting, there is a lack of discussion on reactive power control, which is sometimes crucial to power system reliability. On the other hand, for those references \cite{Ko07,Janssens07,ZhiDW07,Marinescu04,Tarnowski07,WuF08,Pena96,Monroy08,Peresada04,LiDD04,Brekken03,Alegria04,Hopfensperger00,Rabelo01,LiH06} considering only the electrical part, \cite{Janssens07,Tarnowski07} considered only the active power control, while \cite{Ko07,Marinescu04,LiDD04,Brekken03} dealt with the reactive power. In contrast, we provide an integrated solution to the more difficult problem of simultaneously controlling {\em both} the active and reactive powers by appropriately adjusting {\em both} the rotor voltages and the blade pitch angle.
\item From a modeling point of view, the wind turbine model we consider is one of the most comprehensive. Specifically, \cite{Malinga03,Hand99,ZhangL08} assumed a linearized model of the mechanical torque from wind turbines, while \cite{Iyasere08,Prats00,Johnson04b,Johnson06,Beltran08,Calderaro07,Galdi08,Chedid00,GengH09} considered a nonlinear one. In addition, \cite{Ko07,ZhiDW07,Marinescu04,Tarnowski07,WuF08,Pena96,LiDD04,Brekken03,Alegria04,Hopfensperger00,Rabelo01,LiH06,Hansen04} considered a reduced-order DFIG model, while \cite{Peresada04,Monroy08} considered a full-order one. In contrast, we consider a fifth-order, nonlinear, electromechanically-coupled model and attempt no linearization around some operating points.
\item Finally, from a controls point of view, the control techniques we use are different from those adopted in \cite{Prats00,Iyasere08,Johnson04b,Johnson06,Calderaro07,Galdi08,Chedid00,Beltran08,Malinga03,Hand99,ZhangL08,GengH09,Muljadi01,Senjyu06,Stol01,Wright03,Wright03b,ZhangJZ08,Janssens07,Tarnowski07,Ko07,Brekken03,Marinescu04,LiDD04,ZhiDW07,WuF08, Alegria04,Pena96,Hopfensperger00,Monroy08,Peresada04,Hansen04,LiH06,Rabelo01}. For control of the mechanical part of the wind turbines, several techniques have been considered, including proportional-integral-derivative (PID) \cite{Malinga03,Hand99,ZhangL08}, fuzzy control \cite{Chedid00,Prats00,Galdi08,ZhangJZ08,Calderaro07}, adaptive control \cite{Johnson04b,Johnson06}, robust nonlinear control \cite{Iyasere08,GengH09}, sliding mode control \cite{Beltran08}, and wind-model-based predictive control \cite{Senjyu06}. Similarly, for control of the electrical part, various ways of controlling the DFIGs have been proposed, including vector/decoupling control with or without a position encoder \cite{Pena96,Tarnowski07,Ko07,Peresada04,Brekken03,Hopfensperger00,Hansen04,LiH06,Rabelo01}, direct power control \cite{ZhiDW07}, power error vector control \cite{Alegria04}, robust coordinated control \cite{Marinescu04}, linear quadratic regulator \cite{WuF08}, nonlinear inverse system method \cite{LiDD04}, and passivity-based control \cite{Monroy08}. In comparison, the controller we propose here consists of a mixture of several linear/nonlinear control techniques, including feedback linearization, pole placement, and gradient-based optimization and involving three time scales. The reason our controller is more complex is that the problem we address is inherently more challenging (see (i) and (ii)) and the model is more comprehensive (see (i) and (iii)).
\end{enumerate}

As it follows from the above brief review of the current literature, this paper contributes significantly to the state of the art on the control of wind turbines.

The remainder of the paper is organized as follows: Section~\ref{sec:mod} describes a model of variable-speed wind turbines with DFIGs. Section~\ref{sec:contdes} introduces the proposed feedback/feedforward nonlinear controller. Simulation results are shown in Section~\ref{sec:simstud}. Finally, Section~\ref{sec:concl} concludes the paper.

\section{Modeling}\label{sec:mod}

Consider a variable-speed wind turbine consisting of a doubly fed induction generator (DFIG) and a power electronics converter, as shown in Figure~\ref{fig:wtdfig}. The DFIG may be regarded as a slip-ring induction machine, whose stator winding is directly connected to the grid, and whose rotor winding is connected to the grid through a bidirectional frequency converter using back-to-back PWM voltage-source converters.

\begin{figure}[tb]
\centering\includegraphics[width=\linewidth]{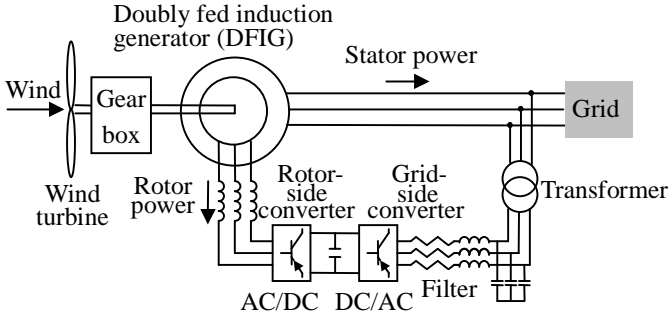}
\caption{Schematic of a typical variable-speed wind turbine with a DFIG.}
\label{fig:wtdfig}
\end{figure}

The dynamics of the electrical part of the wind turbine are represented by a fourth-order state space model, constructed using the synchronously rotating reference frame ($dq$-frame), where the relation between the three phase quantities and the $dq$ components is defined by Park's transformation\cite{Bose02}. The voltage equations are \cite{Fadaeinedjad08}
\begin{align}
v_{ds}&=R_si_{ds}-\omega_s\varphi_{qs}+\frac{d}{dt}\varphi_{ds},\label{eqn:vds}\\
v_{qs}&=R_si_{qs}+\omega_s\varphi_{ds}+\frac{d}{dt}\varphi_{qs},\label{eqn:vqs}\displaybreak[0]\\
v_{dr}&=R_ri_{dr}-(\omega_s-\omega_r)\varphi_{qr}+\frac{d}{dt}\varphi_{dr},\label{eqn:vdr}\\
v_{qr}&=R_ri_{qr}+(\omega_s-\omega_r)\varphi_{dr}+\frac{d}{dt}\varphi_{qr},\label{eqn:vqr}
\end{align}
where $v_{ds}$, $v_{qs}$, $v_{dr}$, $v_{qr}$ are the $d$- and $q$-axis of the stator and rotor voltages; $i_{ds}$, $i_{qs}$, $i_{dr}$, $i_{qr}$ are the $d$- and $q$-axis of the stator and rotor currents; $\varphi_{ds}$, $\varphi_{qs}$, $\varphi_{dr}$, $\varphi_{qr}$ are the $d$- and $q$-axis of the stator and rotor fluxes; $\omega_s$ is the constant angular velocity of the synchronously rotating reference frame; $\omega_r$ is the rotor angular velocity; and $R_s$, $R_r$ are the stator and rotor resistances. The flux equations are \cite{Fadaeinedjad08}
\begin{align}
\varphi_{ds}&=L_si_{ds}+L_mi_{dr},\label{eqn:phids}\\
\varphi_{qs}&=L_si_{qs}+L_mi_{qr},\label{eqn:phiqs}\displaybreak[0]\\
\varphi_{dr}&=L_mi_{ds}+L_ri_{dr},\label{eqn:phidr}\\
\varphi_{qr}&=L_mi_{qs}+L_ri_{qr},\label{eqn:phiqr}
\end{align}
where $L_s$, $L_r$, and $L_m$ are the stator, rotor, and mutual inductances, respectively, satisfying $L_s>L_m$ and $L_r>L_m$.
From \eqref{eqn:phids}--\eqref{eqn:phiqr}, the current equations can be written as
\begin{align}
i_{ds}&=\frac{1}{\sigma L_s}\varphi_{ds}-\frac{L_m}{\sigma L_sL_r}\varphi_{dr},\label{eqn:ids}\\
i_{qs}&=\frac{1}{\sigma L_s}\varphi_{qs}-\frac{L_m}{\sigma L_sL_r}\varphi_{qr},\label{eqn:iqs}\displaybreak[0]\\
i_{dr}&=-\frac{L_m}{\sigma L_sL_r}\varphi_{ds}+\frac{1}{\sigma L_r}\varphi_{dr},\label{eqn:idr}\\
i_{qr}&=-\frac{L_m}{\sigma L_sL_r}\varphi_{qs}+\frac{1}{\sigma L_r}\varphi_{qr},\label{eqn:iqr}
\end{align}
where $\sigma=(1-\frac{L_m^2}{L_sL_r})$ is the leak coefficient. Selecting the fluxes as state variables and substituting \eqref{eqn:ids}--\eqref{eqn:iqr} into \eqref{eqn:vds}--\eqref{eqn:vqr}, the electrical dynamics in state space form can be written as
\begin{align}
\frac{d}{dt}\varphi_{ds}&=-\frac{R_s}{\sigma L_s}\varphi_{ds}+\omega_s\varphi_{qs}+\frac{R_sL_m}{\sigma L_sL_r}\varphi_{dr} +v_{ds},\label{eqn:phids1}\\
\frac{d}{dt}\varphi_{qs}&=-\omega_s\varphi_{ds}-\frac{R_s}{\sigma L_s}\varphi_{qs}+\frac{R_sL_m}{\sigma L_sL_r}\varphi_{qr}+v_{qs},\displaybreak[0]\\
\frac{d}{dt}\varphi_{dr}&=\frac{R_rL_m}{\sigma L_sL_r}\varphi_{ds}-\frac{R_r}{\sigma L_r}\varphi_{dr}+(\omega_s-\omega_r)\varphi_{qr}+v_{dr},\label{eqn:phidr1}\\
\frac{d}{dt}\varphi_{qr}&=\frac{R_rL_m}{\sigma L_sL_r}\varphi_{qs}-(\omega_s-\omega_r)\varphi_{dr}-\frac{R_r}{\sigma L_r}\varphi_{qr}+v_{qr}.\label{eqn:phiqr1}
\end{align}
Neglecting power losses associated with the stator and rotor resistances, the active and reactive stator and rotor powers are given by \cite{LeiYZ06}
\begin{align}
P_s&=-v_{ds}i_{ds}-v_{qs}i_{qs},\label{eqn:Ps}\\
Q_s&=-v_{qs}i_{ds}+v_{ds}i_{qs},\label{eqn:Qs}\displaybreak[0]\\
P_r&=-v_{dr}i_{dr}-v_{qr}i_{qr},\label{eqn:Pr}\\
Q_r&=-v_{qr}i_{dr}+v_{dr}i_{qr},\label{eqn:Qr}
\end{align}
and the total active and reactive powers of the turbine are
\begin{align}
P&=P_s+P_r,\label{eqn:P}\\
Q&=Q_s+Q_r,\label{eqn:Q}
\end{align}
where positive (negative) values of $P$ and $Q$ mean that the turbine injects power into (draws power from) the grid.

The dynamics of the mechanical part of the wind turbine are represented by a first-order model
\begin{align}
J\frac{d}{dt}\omega_r&=T_m-T_e-C_f\omega_r\label{eqn:mech},
\end{align}
where $C_f$ is the friction coefficient, $T_m$ is the mechanical torque generated, and $T_e$ is the electromagnetic torque given by \cite{LeiYZ06}
\begin{align}
T_e=\varphi_{qs}i_{ds}-\varphi_{ds}i_{qs}\label{eqn:Te},
\end{align}
where positive (negative) values mean the turbine acts as a generator (motor). The mechanical power captured by the wind turbine is given by\cite{Bianchi07}
\begin{align}
P_m=T_m\omega_r=\frac{1}{2}\rho AC_p(\lambda,\beta)V_{w}^3,\label{eqn:Pm}
\end{align}
where $\rho$ is the air density; $A=\pi R^2$ is the area swept by the rotor blades of radius $R$; $V_{w}$ is the wind speed; and $C_p(\lambda,\beta)$ is the performance coefficient of the wind turbine, whose value is a function \cite{Bianchi07} of the tip speed ratio $\lambda$, defined as
\begin{align}
\lambda =\frac{\omega_rR}{V_{w}},\label{eqn:lambda}
\end{align}
as well as the blade pitch angle $\beta$, assumed to lie within some mechanical limits $\beta_{\min}$ and $\beta_{\max}$. This function is typically provided by turbine manufacturers and may vary greatly from one turbine to another \cite{Bianchi07}. Therefore, to make the results of this paper broadly applicable to a wide variety of turbines, no specific expression of $C_p(\lambda,\beta)$ will be assumed, until it is absolutely necessary in Section~\ref{sec:simstud}, to carry out simulations. Instead, $C_p(\lambda,\beta)$ will only be assumed to satisfy the following mild conditions:
\begin{enumerate}
\renewcommand{\theenumi}{(A\arabic{enumi})}\itemsep-\parsep
\renewcommand{\labelenumi}{\theenumi}
\item Function $C_p(\lambda,\beta)$ is continuously differentiable in both $\lambda$ and $\beta$ over $\lambda\in(0, \infty)$ and $\beta\in[\beta_{\min}, \beta_{\max}]$.
\item There exists $c\in(0, \infty)$ such that for all $\lambda\in(0, \infty)$ and $\beta \in [\beta_{\min}, \beta_{\max}]$, we have $C_p(\lambda,\beta)\leq c\lambda$. This condition is mild because it is equivalent to saying that the mechanical torque $T_m$ is bounded from above, since $T_m\propto\frac{C_p(\lambda,\beta)}{\lambda}$ according to \eqref{eqn:Pm} and \eqref{eqn:lambda}.
\item For each fixed $\beta\in[\beta_{\min}, \beta_{\max}]$, there exists $\lambda_1\in(0, \infty)$ such that for all $\lambda\in(0, \lambda_1)$, we have $C_p(\lambda, \beta)>0$. This condition is also mild because turbines are designed to capture wind power over a wide range of $\lambda$, including times when $\lambda$ is small.
\end{enumerate}

For the purpose of simulation in Section~\ref{sec:simstud}, the following model of $C_p(\lambda, \beta)$ will be assumed, which is presented in \cite{Heier98} and also adopted in MATLAB/Simulink R2007a:
\begin{align}
C_p(\lambda,\beta)=c_1\Bigl(\frac{c_2}{\lambda_i}-c_3\beta-c_4\Bigl)e^{\frac{-c_5}{\lambda_i}}+c_6\lambda,\label{eqn:cp}
\end{align}
where
\begin{align}
\frac{1}{\lambda_i}=\frac{1}{\lambda+0.08\beta}-\frac{0.035}{\beta^3+1},\label{eqn:lambdai}
\end{align}
and the coefficients are $c_1=0.5176$, $c_2=116$, $c_3=0.4$, $c_4=5$, $c_5=21$, and $c_6=0.0068$.

As it follows from the above, the wind turbine studied in this paper is described by a fifth-order nonlinear dynamical system with states $[\varphi_{ds}\; \varphi_{qs}\; \varphi_{dr}\; \varphi_{qr}\; \omega_r]^T$, controls $[v_{dr}\; v_{qr}\; \beta]^T$, outputs $[P\; Q]^T$, ``disturbance'' $V_w$, nonlinear state equations \eqref{eqn:phids1}--\eqref{eqn:phiqr1} and \eqref{eqn:mech}, and nonlinear output equations \eqref{eqn:Ps}--\eqref{eqn:Q}. Notice that the system dynamics are strongly coupled: the ``mechanical'' state variable $\omega_r$ affects the electrical dynamics bilinearly via \eqref{eqn:phidr1} and \eqref{eqn:phiqr1}, while the ``electrical'' state variables $[\varphi_{ds}\; \varphi_{qs}\; \varphi_{dr}\; \varphi_{qr}]^T$ affect the mechanical dynamics quadratically via \eqref{eqn:ids}--\eqref{eqn:iqr}, \eqref{eqn:Te} and \eqref{eqn:mech}. Since the stator winding of the DFIG is directly connected to the grid, for reliability reasons $[v_{ds}\; v_{qs}]^T$ are assumed to be fixed, i.e., not to be controlled, in the rest of this paper. Moreover, since \eqref{eqn:ids}--\eqref{eqn:iqr} represent a bijective mapping between $[\varphi_{ds}\; \varphi_{qs}\; \varphi_{dr}\; \varphi_{qr}]^T$ and $[i_{ds}\; i_{qs}\; i_{dr}\; i_{qr}]^T$ and since the currents $[i_{ds}\; i_{qs}\; i_{dr}\; i_{qr}]^T$, the rotor angular velocity $\omega_r$, and the wind speed $V_w$ can all be measured, a controller for this system has access to its entire states (i.e., full state feedback is available) and its disturbance (i.e., the wind speed $V_w$). A block diagram of this system is shown on the right-hand side of Figure~\ref{fig:wtc}.

\begin{figure*}[tb]
\centering\includegraphics[width=0.8\textwidth]{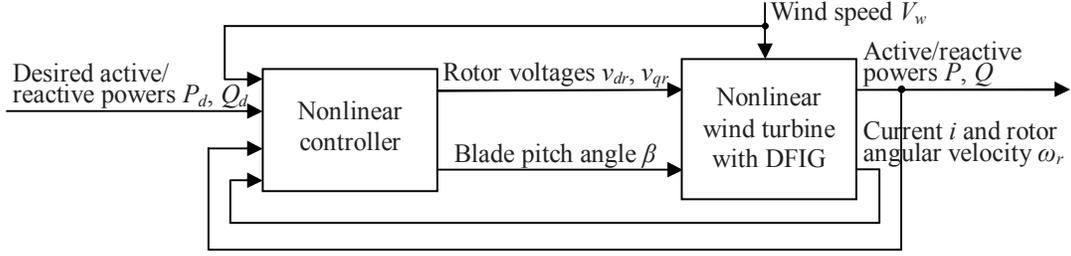}
\caption{Structure of the multivariable, feedback/feedforward nonlinear controller, developed based on original, nonlinear dynamics of the wind turbine.}
\label{fig:wtc}
\end{figure*}

\section{Controller Design}\label{sec:contdes}

In this section, a feedback/feedforward nonlinear controller of the form depicted on the left-hand side of Figure~\ref{fig:wtc} is presented. By adjusting the rotor voltages $v_{dr}$ and $v_{qr}$ and the blade pitch angle $\beta$, the controller attempts to make the active and reactive powers $P$ and $Q$ track, as closely as possible---limited only by wind strength---some desired, time-varying references $P_d$ and $Q_d$, presumably provided by a wind farm operator. When $P_d$ is set to sufficiently large, i.e., larger than what the turbine can possibly convert from wind, it means the operator wants the turbine to operate in the MPT mode; otherwise, the PR mode is sought. The value of $Q_d$, along with that of $P_d$, reflects a desired power factor $\text{PF}_d=\frac{P_d}{\sqrt{P_d^2+Q_d^2}}$ the operator wants the turbine to also maintain.

The controller development consists of three steps, which are described in Sections~\ref{ssec:feedlinepoleplace}--\ref{ssec:gradappr}, respectively.

\subsection{Feedback Linearization and Pole Placement}\label{ssec:feedlinepoleplace}

For convenience, let us introduce the variable $x=[x_1\; x_2\; x_3\; x_4]^T=[\varphi_{ds}\; \varphi_{qs}\; \varphi_{dr}\; \varphi_{qr}]^T$ and rewrite \eqref{eqn:ids}--\eqref{eqn:iqr} and \eqref{eqn:phids1}--\eqref{eqn:phiqr1} in matrix forms as follows:
\begin{align}
\dot{x}
&=\underbrace{
\begin{bmatrix}
-\frac{R_s}{\sigma L_s} & \omega_s & \frac{R_sL_m}{\sigma L_sL_r} & 0 \\
-\omega_s & -\frac{R_s}{\sigma L_s} & 0 & \frac{R_sL_m}{\sigma L_sL_r}\\
\frac{R_rL_m}{\sigma L_sL_r} & 0 & -\frac{R_r}{\sigma L_r} & \omega_s \\
0 & \frac{R_rL_m}{\sigma L_sL_r} & -\omega_s & -\frac{R_r}{\sigma L_r}
\end{bmatrix}}_{A}x\nonumber\\
&\qquad+
\begin{bmatrix}
v_{ds} \\ v_{qs} \\ v_{dr}-\omega_r x_4 \\ v_{qr}+\omega_r x_3
\end{bmatrix},\label{eqn:statespacemodeldfig}\displaybreak[0]\\
\begin{bmatrix}
i_{ds} \\ i_{qs} \\ i_{dr} \\ i_{qr}
\end{bmatrix}
&=
\begin{bmatrix}
\frac{1}{\sigma L_s} & 0 & -\frac{L_m}{\sigma L_sL_r} & 0 \\
0 & \frac{1}{\sigma L_s} & 0 & -\frac{L_m}{\sigma L_sL_r} \\
-\frac{L_m}{\sigma L_sL_r} & 0 & \frac{1}{\sigma L_r} & 0 \\
0 & -\frac{L_m}{\sigma L_sL_r} & 0 & \frac{1}{\sigma L_r}
\end{bmatrix}x,\label{eqn:output}
\end{align}
where $A$ is a constant matrix and, as was pointed out at the end of Section \ref{sec:mod}, both $v_{ds}$ and $v_{qs}$ are constants not to be controlled. Note that the only nonlinearities in \eqref{eqn:statespacemodeldfig} are the two products of the state variables, i.e., $-\omega_rx_4$ and $\omega_rx_3$. Also note that these nonlinearities appear on the same rows as the control variables $v_{dr}$ and $v_{qr}$. Thus, {\em feedback linearization} \cite{Khalil01} may be used to cancel them and subsequently perform arbitrary {\em pole placement} \cite{Chen99}, i.e., let
\begin{align}
v_{dr}&=\omega_r x_4-K_1^Tx+u_1, \label{eqn:vdr1}\\
v_{qr}&=-\omega_r x_3-K_2^Tx+u_2,\label{eqn:vqr1}
\end{align}
where $K_1, K_2 \in \mathbb{R}^4$, the first terms on the right-hand side of \eqref{eqn:vdr1} and \eqref{eqn:vqr1} are intended to cancel the nonlinearities, the second terms are for pole placement, and the third are new control variables $u_1$ and $u_2$, to be designed later.

To implement \eqref{eqn:vdr1} and \eqref{eqn:vqr1}, full state feedback on the fluxes $x$ and the rotor angular velocity $\omega_r$ are needed. While the latter is relatively easy to measure, the former is not. Fortunately, this difficulty can be circumvented by first measuring the currents---which is feasible---and then calculating the fluxes from \eqref{eqn:phids}--\eqref{eqn:phiqr}. This explains the fourth input of the {\it nonlinear controller} block in Figure~\ref{fig:wtc}.

Substituting \eqref{eqn:vdr1} and \eqref{eqn:vqr1} into \eqref{eqn:statespacemodeldfig} yields
\begin{align}
\dot{x}=(A-BK)x+\begin{bmatrix}
v_{ds}\;v_{qs}\;u_1\;u_2
\end{bmatrix}^T, \label{eqn:dotx}
\end{align}
where $B=[0_{2\times 2}\quad I_{2\times 2}]^T$ and $K=[K_1\quad K_2]^T$ is the state feedback gain matrix.
Since the electrical elements in the DFIG are physically allowed to have much faster responses than their mechanical counterparts, $K$ in \eqref{eqn:dotx} may be chosen so that $A-BK$ is asymptotically stable with very fast eigenvalues. With this choice of $K$ and with relatively slow-varying $u_1$ and $u_2$ (recall that $v_{ds}$ and $v_{qs}$ are constants), the fourth-order linear differential equation \eqref{eqn:dotx} may be approximated by the following static, linear equation:
\begin{align}
x=-(A-BK)^{-1}
\begin{bmatrix}
v_{ds}\;v_{qs}\;u_1\;u_2
\end{bmatrix}^T.\label{eqn:equilibrium}
\end{align}
As a result, the fifth-order model described in \eqref{eqn:phids1}--\eqref{eqn:phiqr1} and \eqref{eqn:mech} may be approximated by the first-order model described in \eqref{eqn:mech} along with \eqref{eqn:equilibrium}. As will be shown next, this approximation greatly simplifies the design of $u_1$ and $u_2$. Therefore, we will assume, in the sequel, that $K$ is chosen so that the electrical dynamics \eqref{eqn:dotx} are asymptotically stable and so fast that they may be approximated by \eqref{eqn:equilibrium}.

\subsection{Tracking of Desired Angular Velocity}\label{sec:asymtracdesangvel}

The second step of the controller development involves constructing a speed controller that ensures the angular velocity of the rotor, $\omega_r$, tracks a desired, time-varying reference, $\omega_{rd}$, whenever possible. The construction may be divided into four substeps as described below.

{\bf Substep 1.} First, we show that the electromagnetic torque $T_e$ defined in \eqref{eqn:Te} may be expressed as a {\em quadratic function} of the new control variables $u_1$ and $u_2$. From \eqref{eqn:Te} and \eqref{eqn:output},
\begin{align}
T_e&=\varphi_{qs}i_{ds}-\varphi_{ds}i_{qs}=
\begin{bmatrix}
x_2 & -x_1
\end{bmatrix}
\begin{bmatrix}
i_{ds} \\ i_{qs}
\end{bmatrix}\nonumber\displaybreak[0]\\
&=x^T
\begin{bmatrix}
0 & -\frac{1}{\sigma L_s} & 0 & \frac{L_m}{\sigma L_sL_r} \\
\frac{1}{\sigma L_s} & 0 & -\frac{L_m}{\sigma L_sL_r} & 0 \\
0 & 0 & 0 & 0 \\
0 & 0 & 0 & 0
\end{bmatrix}
x.\label{eqn:torquequadraticx}
\end{align}
Equation \eqref{eqn:torquequadraticx} suggests that $T_e$ is a quadratic function of $x$, while \eqref{eqn:equilibrium} suggests that $x$, in turn, is an affine function of $u_1$ and $u_2$, since $v_{ds}$ and $v_{qs}$ in \eqref{eqn:equilibrium} are assumed to be constants. Hence, $T_e$ must be a quadratic function of $u_1$ and $u_2$. Indeed, an explicit expression can be obtained as follows: since $A-BK$ is asymptotically stable and thus nonsingular, it may be written as
\begin{align}
(A-BK)^{-1}=
\begin{bmatrix}
d_{11} & d_{12} & d_{13} & d_{14} \\
d_{21} & d_{22} & d_{23} & d_{24} \\
d_{31} & d_{32} & d_{33} & d_{34} \\
d_{41} & d_{42} & d_{43} & d_{44}
\end{bmatrix},\label{eqn:convenience}
\end{align}
where each $d_{ij}$ depends on $A$, $B$, and $K$. From \eqref{eqn:equilibrium}--\eqref{eqn:convenience},
\begin{align}
T_e
&=\frac{L_m}{\sigma L_sL_r}[(d_{11}v_{ds}+d_{12}v_{qs}+d_{13}u_1+d_{14}u_2)\nonumber\\
&\qquad(d_{41}v_{ds}+d_{42}v_{qs}+d_{43}u_1+d_{44}u_2)\nonumber\\
&\qquad-(d_{21}v_{ds}+d_{22}v_{qs}+d_{23}u_1+d_{24}u_2)\nonumber\\
&\qquad(d_{31}v_{ds}+d_{32}v_{qs}+d_{33}u_1+d_{34}u_2)]\nonumber\displaybreak[0]\\
&=\begin{bmatrix}
u_1 & u_2
\end{bmatrix}
\begin{bmatrix}
q_1 & q_2 \\ q_2 & q_3
\end{bmatrix}
\begin{bmatrix}
u_1 \\ u_2
\end{bmatrix}+
\begin{bmatrix}
b_1 & b_2
\end{bmatrix}
\begin{bmatrix}
u_1 \\ u_2
\end{bmatrix}
+a,\label{eqn:Tequa}
\end{align}
where $q_1$, $q_2$, $q_3$, $b_1$, $b_2$, and $a$ are constants defined as
\begin{align}
q_1&=\frac{L_m}{\sigma L_sL_r}(d_{13}d_{43}-d_{23}d_{33}),\label{eqn:q1}\displaybreak[0]\\
q_2&=\frac{1}{2}\frac{L_m}{\sigma L_sL_r}(d_{13}d_{44}+d_{14}d_{43}-d_{23}d_{34}-d_{24}d_{33}),\label{eqn:q2}\displaybreak[0]\\
q_3&=\frac{L_m}{\sigma L_sL_r}(d_{14}d_{44}-d_{24}d_{34}),\label{eqn:q3}\displaybreak[0]\\
b_1\!&=\!\frac{L_m}{\sigma L_sL_r}\big((d_{11}v_{ds}\!\!+\!d_{12}v_{qs})d_{43}\!\!+\!d_{13}(d_{41}v_{ds}\!\!+\!d_{42}v_{qs})\nonumber\\
&\quad-(d_{21}v_{ds}+d_{22}v_{qs})d_{33}-d_{23}(d_{31}v_{ds}+d_{32}v_{qs})\big),\label{eqn:b1}\displaybreak[0]\\
b_2\!&=\!\frac{L_m}{\sigma L_sL_r}\big((d_{11}v_{ds}\!\!+\!d_{12}v_{qs})d_{44}\!\!+\!d_{14}(d_{41}v_{ds}\!\!+\!d_{42}v_{qs})\nonumber\\
&\quad-(d_{21}v_{ds}+d_{22}v_{qs})d_{34}-d_{24}(d_{31}v_{ds}+d_{32}v_{qs})\big),\label{eqn:b2}\displaybreak[0]\\
a&=\frac{L_m}{\sigma L_sL_r}\big((d_{11}v_{ds}+d_{12}v_{qs})(d_{41}v_{ds}+d_{42}v_{qs})\nonumber\\
&\quad-(d_{21}v_{ds}+d_{22}v_{qs})(d_{31}v_{ds}+d_{32}v_{qs})\big).\label{eqn:a}
\end{align}

{\bf Substep 2.} Next, we show that the quadratic function \eqref{eqn:Tequa} relating $u_1$ and $u_2$ to $T_e$ has a desirable feature: its associated Hessian matrix $\left[\begin{smallmatrix}
q_1 & q_2 \\
q_2 & q_3
\end{smallmatrix}\right]$
is always {\em positive definite}, regardless of the parameters of the electrical part of the DFIG, as well as the choice of the state feedback gain matrix $K$. The following lemma formally states and proves this assertion:

\begin{lemma} \label{lemma:posdef} The Hessian matrix
$\left[\begin{smallmatrix}
q_1 & q_2 \\
q_2 & q_3
\end{smallmatrix}\right]$ in \eqref{eqn:Tequa}
is positive definite.
\end{lemma}

\begin{proof}
From \eqref{eqn:statespacemodeldfig}, \eqref{eqn:dotx}, \eqref{eqn:convenience}, \eqref{eqn:q1}, \eqref{eqn:q2}, and \eqref{eqn:q3}, the determinant of $A-BK$ and the leading principal minors $q_1$ and $q_1q_3-q_2^2$ of $\left[\begin{smallmatrix}
q_1 & q_2 \\
q_2 & q_3
\end{smallmatrix}\right]$
can be written as
\begin{align}
|A-BK|&=\frac{\Delta}{(L_sL_r-L_m^2)^2}, \label{eqn:det(A-BK)}\displaybreak[0]\\
q_1&=\frac{R_sL_m^2(\Delta_1^2+\Delta_2^2)}{\Delta^2}, \label{eqn:q11}\displaybreak[0]\\
q_1q_3-q_2^2&=\frac{R_s^2L_m^4}{\Delta^2}, \label{eqn:q4}
\end{align}
where
\begin{align}
\Delta&=\big(-R_sL_mk_{12}+(L_sL_r-L_m^2)k_{13}-R_sL_rk_{14}\nonumber\\
&\qquad+R_rL_s+R_sL_r\big)\Delta_1 \nonumber\\
&\quad+\big(R_sL_mk_{11}+R_sL_rk_{13}+(L_sL_r-L_m^2)k_{14}\nonumber\\
&\qquad+R_sR_r-L_sL_r+L_m^2\big)\Delta_2, \label{eqn:delta}\displaybreak[0]\\
\Delta_1&=R_sL_mk_{21}+R_sL_rk_{23}+(L_sL_r-L_m^2)k_{24}\nonumber\\
&\qquad+R_rL_s+R_sL_r,\nonumber\displaybreak[0]\\
\Delta_2&=R_sL_mk_{22}+(-L_sL_r+L_m^2)k_{23}+R_sL_rk_{24}\nonumber\\
&\qquad+R_sR_r-L_sL_r+L_m^2,\nonumber
\end{align}
and $k_{ij}$ is the $ij$ entry of $K$. Since $A-BK$ is nonsingular, $L_s>L_m$, and $L_r>L_m$, it follows from \eqref{eqn:det(A-BK)} that $\Delta\neq0$. Since $\Delta\neq0$, it follows from \eqref{eqn:q4} that $q_1q_3-q_2^2>0$ and from \eqref{eqn:delta} that $\Delta_1$ and $\Delta_2$ cannot be zero simultaneously. The latter, along with \eqref{eqn:q11}, implies that $q_1>0$. Since $q_1>0$ and $q_1q_3-q_2^2>0$, $\left[\begin{smallmatrix}
q_1 & q_2 \\
q_2 & q_3
\end{smallmatrix}\right]$ in \eqref{eqn:Tequa} must be positive definite.
\end{proof}

{\bf Substep 3.} Next, we show that there is a {\em redundancy} in the control variables $u_1$ and $u_2$, which may be exposed via a {\em coordinate change}. Observe from \eqref{eqn:mech} that the first-order dynamics of $\omega_r$ are driven by $T_e$. Also observe from Substeps~1 and~2 that $T_e$ is a convex quadratic function of $u_1$ and $u_2$. Thus, we have {\em two} coupled control inputs (i.e., $u_1$ and $u_2$) affecting {\em one} state variable (i.e., $\omega_r$), implying that there is a redundancy in the control inputs, which may be exploited elsewhere (to be discussed in Section~\ref{ssec:gradappr}). To expose this redundancy, first notice that because the Hessian matrix $\left[\begin{smallmatrix}
q_1 & q_2 \\
q_2 & q_3
\end{smallmatrix}\right]$ is positive definite, it can be diagonalized, i.e., there exist an orthogonal matrix $M$ containing its eigenvectors and a diagonal matrix $D$ containing its eigenvalues, such that
\begin{align}
M^T
\begin{bmatrix}
q_1 & q_2 \\
q_2 & q_3
\end{bmatrix}M=D.\label{eqn:diagonalization}
\end{align}
Indeed,
\begin{align*}
M=\begin{bmatrix}
\frac{q_2}{\sqrt{q_2^2+(\lambda_1-q_1)^2}} & \frac{\lambda_2-q_3}{\sqrt{(\lambda_2-q_3)^2+q_2^2}}\\
\frac{\lambda_1-q_1}{\sqrt{q_2^2+(\lambda_1-q_1)^2}} & \frac{q_2}{\sqrt{(\lambda_2-q_3)^2+q_2^2}}
\end{bmatrix},
\quad D=\begin{bmatrix}
\lambda_1 & 0\\
0 & \lambda_2
\end{bmatrix},
\end{align*}
where \begin{align*}
\lambda_{1,2}=\frac{q_1+q_3\pm\sqrt{(q_1+q_3)^2-4(q_1q_3-q_2^2)}}{2}.
\end{align*}
Next, consider the following coordinate change, which transforms $u_1 \in \mathbb{R}$ and $u_2 \in \mathbb{R}$ in a Cartesian coordinate system into $r \geq 0$ and $\theta \in [0, 2\pi)$ in a polar coordinate system:
\begin{align}
r=\sqrt{z^Tz},\quad\theta=\tan^{-1}\big(\frac{z_2}{z_1}\big),\label{eqn:ztheta}
\end{align}
where
\begin{align}
z&=
\begin{bmatrix}
z_1 \\ z_2
\end{bmatrix}=D^{1/2}M^T\begin{bmatrix}
u_1 \\ u_2 \end{bmatrix}+\frac{1}{2}D^{-1/2}M^T\begin{bmatrix}
b_1 \\ b_2 \end{bmatrix}. \label{eqn:zfunu}
\end{align}
In terms of the new coordinates $r$ and $\theta$, it follows from \eqref{eqn:Tequa} and \eqref{eqn:diagonalization}--\eqref{eqn:zfunu} that
\begin{align}
T_e&=r^2+a',\label{eqn:Techange2}
\end{align}
where
\begin{align*}
a'&=a-\frac{1}{4}
\begin{bmatrix} b_1 & b_2 \end{bmatrix}
\begin{bmatrix}
q_1 & q_2 \\
q_2 & q_3
\end{bmatrix}^{-1}
\begin{bmatrix} b_1 \\ b_2 \end{bmatrix}.
\end{align*}
Using \eqref{eqn:q1}--\eqref{eqn:a}, $a'$ may be simplified to
\begin{align}
a'=-\frac{v_{ds}^2+v_{qs}^2}{4\omega_sR_s},\label{eqn:a1}
\end{align}
implying that it is always negative. Comparing \eqref{eqn:Techange2} with \eqref{eqn:Tequa} shows that the coordinate change \eqref{eqn:ztheta} and \eqref{eqn:zfunu} allows us to decouple the control variables, so that in the new coordinates, $r$ is responsible for driving the first-order dynamics of $\omega_r$ through $T_e$ of \eqref{eqn:Techange2}, while $\theta$ does not at all affect $\omega_r$ (and, hence, is redundant as far as the dynamics of $\omega_r$ are concerned). The design of $r$ and $\theta$ will be discussed in Substep 4 and Section \ref{ssec:gradappr}, respectively.

{\bf Substep 4.} Finally, a {\em speed controller} is presented, which ensures that the rotor angular velocity $\omega_r$ tracks a desired time-varying reference $\omega_{rd}$, to be determined in Section~\ref{ssec:gradappr}, provided that $\omega_{rd}$ is not exceedingly large. Combining \eqref{eqn:mech} and \eqref{eqn:Techange2} yields
\begin{align}
J\dot{\omega}_r=T_m(\omega_r, \beta, V_w)-r^2-a'-C_f\omega_r, \label{eqn:mechTechange}
\end{align}
where, according to \eqref{eqn:Pm},
\begin{align}
T_m(\omega_r, \beta, V_w)=\frac{\frac{1}{2}\rho AC_p(\lambda, \beta)V_w^3}{\omega_r}. \label{eqn:Tm}
\end{align}
Here, $T_m$ is written as $T_m(\omega_r, \beta, V_w)$ to emphasize its dependence on $\omega_r$, $\beta$, and $V_w$. Observe from \eqref{eqn:mechTechange} that, if the control input $r^2$ were {\em real-valued} instead of being nonnegative, feedback linearization may be applied to cancel all the terms on the right-hand side of \eqref{eqn:mechTechange} and insert linear dynamics $\alpha (\omega_r-\omega_{rd})$, i.e., we may let
\begin{align}
r^2=T_m(\omega_r, \beta, V_w)-a'-C_f\omega_r+\alpha(\omega_r-\omega_{rd}),\label{eqn:r2imaginary}
\end{align}
so that
\begin{align}
J\dot{\omega}_r=-\alpha(\omega_r-\omega_{rd}).\label{eqn:lindyn}
\end{align}
By letting the controller parameter $\alpha$ be positive, \eqref{eqn:lindyn} implies that $\omega_r$ always attempts to go to $\omega_{rd}$. Unfortunately, because $r^2$ cannot be negative, the speed controller \eqref{eqn:r2imaginary}---and, hence, the linear dynamics \eqref{eqn:lindyn}---cannot be realized whenever the right-hand side of \eqref{eqn:r2imaginary} is negative. To alleviate this issue, \eqref{eqn:r2imaginary} is slightly modified by setting $r^2$ to zero whenever that occurs, i.e.,
\begin{align}
r^2\!=\!\max\{T_m(\omega_r, \beta, V_w)\!-\!a'\!-\!C_f\omega_r\!+\!\alpha (\omega_r\!-\!\omega_{rd}), 0\}.\label{eqn:r}
\end{align}
Notice that \eqref{eqn:r} contains a feedforward action involving the ``disturbance'', i.e., the wind speed $V_w$. This explains the first input of the {\it nonlinear controller} block in Figure~\ref{fig:wtc}.

To analyze the behavior of the speed controller \eqref{eqn:r}, suppose $\omega_{rd}$, $\beta$, and $V_w$ are constants and consider the function $g$, defined as
\begin{align}
g(\omega_r, \beta, V_w)=T_m(\omega_r, \beta, V_w)-a'-C_f\omega_r.\label{eqn:g}
\end{align}
The following lemma says that $g(\omega_r, \beta, V_w)$, when viewed as a function of $\omega_r$, has a positive root $\omega_r^{(1)}$, below which $g(\omega_r, \beta, V_w)$ is positive:

\begin{lemma}\label{lemma:omegar1}
For each fixed $\beta\in[\beta_{\min}, \beta_{\max}]$ and $V_w>0$, there exists $\omega_r^{(1)}\in(0, \infty)$ such that $g(\omega_r^{(1)},\beta,V_w)=0$ and $g(\omega_r,\beta,V_w)>0$ for all $\omega_r\in(0,\omega_r^{(1)})$.
\end{lemma}

\begin{proof}
Due to the fact that $a'$ in \eqref{eqn:a1} is negative, there exists $\omega_{r,1}$ such that $-a'-C_f\omega_r>0$ for all $\omega_r\in(0, \omega_{r,1})$. Due to Assumption (A3) of Section~\ref{sec:mod}, \eqref{eqn:lambda}, and \eqref{eqn:Tm}, there exists $\omega_{r,2}$ such that $T_m(\omega_r, \beta, V_w)>0$ for all $\omega_r\in(0, \omega_{r,2})$. Hence, from \eqref{eqn:g}, we have $g(\omega_r, \beta, V_w)>0$ for all $\omega_r\in(0, \min\{\omega_{r,1}, \omega_{r,2}\})$. In addition, due to Assumption (A2), \eqref{eqn:lambda}, \eqref{eqn:Tm}, \eqref{eqn:g}, and $a'<0$, there exists $\omega_{r,3}$, sufficiently large, such that $g(\omega_{r,3}, \beta, V_w)<0$. These two properties of $g$, along with Assumption (A1) and the Intermediate Value Theorem, imply that there exists at least one positive root $\omega_r$ satisfying $g(\omega_r, \beta, V_w)=0$. Letting $\omega_r^{(1)}$ be the first of such roots completes the proof.
\end{proof}

The following theorem, derived based on Lemma~\ref{lemma:omegar1}, says that as long as the desired rotor angular velocity $\omega_{rd}$ is not exceedingly large, i.e., does not exceed the first root $\omega_r^{(1)}$ of $g(\omega_r, \beta, V_w)$, the closed-loop dynamics \eqref{eqn:mechTechange} and \eqref{eqn:r} have an asymptotically stable equilibrium point at $\omega_{rd}$:

\begin{theorem} \label{thm:speedctrl} Consider the first-order dynamics \eqref{eqn:mechTechange} and the speed controller \eqref{eqn:r}. Suppose $\omega_{rd}$, $\beta$, and $V_w$ are constants, with $\omega_{rd}$ satisfying $0<\omega_{rd}<\omega_r^{(1)}$. Then, for all $\omega_r(0)>0$, $\lim_{t \rightarrow \infty}\omega_r(t)=\omega_{rd}$.
\end{theorem}

\begin{proof}
Substituting \eqref{eqn:r} into \eqref{eqn:mechTechange} and using \eqref{eqn:g} yield
\begin{align}
J\dot{\omega}_r&=\min\{\alpha (\omega_{rd}-\omega_r),g(\omega_r, \beta, V_w)\}.\label{eqn:mechctrl}
\end{align}
Suppose $0<\omega_{rd}<\omega_r^{(1)}$. We first show that $\omega_r=\omega_{rd}$ is the unique equilibrium point of \eqref{eqn:mechctrl}. Suppose $\omega_r=\omega_{rd}$. Then, $\alpha(\omega_{rd}-\omega_r)$ in \eqref{eqn:mechctrl} is zero, whereas $g(\omega_r, \beta, V_w)$ in \eqref{eqn:mechctrl} is positive, due to Lemma~\ref{lemma:omegar1}. Thus, $\dot{\omega}_r=0$, implying that $\omega_{rd}$ is an equilibrium point. Next, suppose $0<\omega_r<\omega_{rd}$. Then, $\alpha(\omega_{rd}-\omega_r)$ is positive, and so is $g(\omega_r, \beta, V_w)$, due again to Lemma~\ref{lemma:omegar1}. Hence, $\dot{\omega}_r>0$, implying that there is no equilibrium point to the left of $\omega_{rd}$. Finally, suppose $\omega_r>\omega_{rd}$. Then, $\alpha(\omega_{rd}-\omega_r)$ is negative. Therefore, $\dot{\omega}_r<0$, implying that there is no equilibrium point to the right of $\omega_{rd}$. From the above analysis, we see that $\omega_r=\omega_{rd}$ is the unique equilibrium point of \eqref{eqn:mechctrl}. Next, we show that the equilibrium point $\omega_r=\omega_{rd}$ is asymptotically stable in that for all $\omega_r(0)>0$, $\lim_{t\rightarrow\infty}\omega_r(t)=\omega_{rd}$. Consider a quadratic Lyapunov function candidate $V:(0, \infty)\rightarrow\mathbb{R}$, defined as
\begin{align}
V(\omega_r)=\frac{1}{2}(\omega_r-\omega_{rd})^2,\label{eqn:LyapuStab}
\end{align}
which is positive definite with respect to the shifted origin $\omega_r=\omega_{rd}$. From \eqref{eqn:mechctrl} and \eqref{eqn:LyapuStab},
\begin{align}
\dot{V}(\omega_r)=\frac{1}{J}(\omega_r-\omega_{rd})\min\{\alpha (\omega_{rd}-\omega_r),g(\omega_r, \beta, V_w)\}. \label{eqn:Vdot}
\end{align}
Note that whenever $0<\omega_r<\omega_{rd}$, $\alpha(\omega_{rd}-\omega_r)>0$ and $g(\omega_r, \beta, V_w)>0$, so that $\dot{V}(\omega_r)<0$ according to \eqref{eqn:Vdot}. On the other hand, whenever $\omega_r>\omega_{rd}$, $\alpha(\omega_{rd}-\omega_r)<0$, so that $\dot{V}(\omega_r)<0$. Finally, when $\omega_r=\omega_{rd}$, $\dot{V}(\omega_r)=0$. Therefore, $\dot{V}(\omega_r)$ is negative definite with respect to the shifted origin $\omega_r=\omega_{rd}$. It follows from \cite{Khalil01} that $\omega_r=\omega_{rd}$ is asymptotically stable, i.e., for all $\omega_r(0)>0$, $\lim_{t \rightarrow \infty}\omega_r(t)=\omega_{rd}$.
\end{proof}

Theorem~\ref{thm:speedctrl} says that the first root $\omega_r^{(1)}$ is a {\em critical root}, for which $\omega_{rd}$ should never exceed, if we want $\omega_r(t)$ to go to $\omega_{rd}$ regardless of $\omega_r(0)$. Figure~\ref{fig:wr1} shows, for the MATLAB/Simulink R2007a model of $C_p(\lambda,\beta)$ given in \eqref{eqn:cp} and \eqref{eqn:lambdai}, how the critical root $\omega_r^{(1)}$ depends on $\beta$ and $V_w$. Notice from the figure that $\omega_r^{(1)}$ is insensitive to $\beta$ but proportional to $V_w$, meaning that the larger the wind speed, the higher the ``ceiling'' on the desired rotor angular velocity. Also notice that $\omega_r^{(1)}$ of more than 3500 in the per-unit system is extremely large, meaning that for this particular turbine there is no need to be concerned about $\omega_{rd}$ exceeding $\omega_r^{(1)}$.

\begin{figure}[tb]
\centering\includegraphics[width=0.95\linewidth]{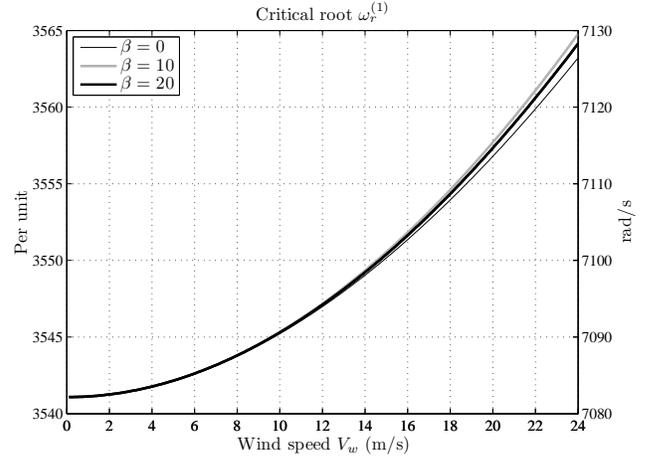}
\caption{Critical root $\omega_r^{(1)}$ as a function of blade pitch angle $\beta$ and wind speed $V_w$.}
\label{fig:wr1}
\end{figure}

\subsection{Lyapunov-like Function and Gradient-based Approach}\label{ssec:gradappr}

The third and final step of the controller development involves introducing a Lyapunov-like function, which measures the difference between the actual and desired powers, and utilizing a gradient-based approach, which minimizes this function.

Recall from the beginning of Section~\ref{sec:contdes} that the objective of the controller is to make the active and reactive powers, $P$ and $Q$, track some desired references, $P_d$ and $Q_d$, as closely as possible. In the MPT mode, where the goal is to generate as much active power as possible while maintaining an acceptable power factor, $P_d$ is set to a value that far exceeds what the wind turbine can possibly produce (e.g., in the per-unit system, $P_d>1$), while $Q_d$ is set to a value representing the desired power factor $\text{PF}_d=\frac{P_d}{\sqrt{P_d^2+Q_d^2}}$. In this mode, making $P$ and $Q$ approach $P_d$ and $Q_d$ is equivalent to maximizing the active power output while preserving the power factor. In the PR mode, where the goal is to regulate the powers, both $P_d$ and $Q_d$ are set to values representing power demands from the grid. In this mode, making $P$ and $Q$ approach $P_d$ and $Q_d$ amounts to achieving power regulation. Hence, the values of $P_d$ and $Q_d$ reflect the mode the wind farm operator wants the wind turbine to operate in. However, as far as the controller is concerned, it does not distinguish between the two modes; all it does is try its best to drive $P$ and $Q$ to $P_d$ and $Q_d$.

To mathematically describe the aforementioned controller objective, consider the following positive definite, quadratic Lyapunov-like function $V$ of the differences $P-P_d$ and $Q-Q_d$:
\begin{align}
V=\frac{1}{2}\begin{bmatrix} P-P_d & Q-Q_d \end{bmatrix}
\underbrace{\begin{bmatrix}
w_p & w_{pq} \\
w_{pq} & w_q
\end{bmatrix}}_{>0}
\begin{bmatrix}
P-P_d \\
Q-Q_d
\end{bmatrix},\label{eqn:objfun}
\end{align}
where $w_p$, $w_q$, and $w_{pq}$ are design parameters that allow one to specify how the differences $P-P_d$ and $Q-Q_d$, as well as their correlation $(P-P_d)(Q-Q_d)$, should be penalized. With this $V$, the above controller objective can be restated simply as: {\em make $V$ go to zero}, because when this happens, $P$ and $Q$ must both go to $P_d$ and $Q_d$. Since it is not always possible to achieve this objective---due to the fact that the wind may not always be strong enough---below we will attempt instead to {\em make $V$ as small as possible by minimizing it}.

To minimize $V$, we first show that $V$ is a function of $\omega_{rd}$, $\theta$, $\beta$, $V_w$, $P_d$, and $Q_d$, i.e.,
\begin{align}
V=f(\omega_{rd}, \theta, \beta, V_w, P_d, Q_d) \label{eqn:Vf}
\end{align}
for some $f$. Note from \eqref{eqn:objfun} that $V$ depends on $P$, $Q$, $P_d$, and $Q_d$. Also note from \eqref{eqn:Ps}--\eqref{eqn:Qr}, \eqref{eqn:P}, and \eqref{eqn:Q} that $P$ and $Q$, in turn, depend on $i$, $v_{dr}$, and $v_{qr}$ (recall that $v_{ds}$ and $v_{qs}$ are constants). Thus,
\begin{align}
V=f_1(i, v_{dr}, v_{qr}, P_d, Q_d) \label{eqn:Vf1}
\end{align}
for some $f_1$. Next, note from \eqref{eqn:ids}--\eqref{eqn:iqr} that $i$ depends on $x$; from \eqref{eqn:vdr1} and \eqref{eqn:vqr1} that $v_{dr}$ and $v_{qr}$ depend on $x$, $\omega_r$, $u_1$, and $u_2$; and from \eqref{eqn:equilibrium} that $x$ further depends on $u_1$ and $u_2$. Hence,
\begin{align}
(i, v_{dr}, v_{qr})=f_2(\omega_r, u_1, u_2) \label{eqn:Vf2}
\end{align}
for some $f_2$. Furthermore, note from \eqref{eqn:ztheta} and \eqref{eqn:zfunu} that $u_1$ and $u_2$ depend on $r$ and $\theta$, where $r$, in turn, depends on $\omega_r$, $\omega_{rd}$, $\beta$, and $V_w$ through \eqref{eqn:r}. Therefore,
\begin{align}
(u_1, u_2)=f_3(\omega_r, \omega_{rd}, \theta, \beta, V_w) \label{eqn:Vf3}
\end{align}
for some $f_3$. Finally, assuming that $\omega_{rd}$ does not exceed the first root $\omega_r^{(1)}$ and assuming that $\omega_{rd}$, $\beta$, and $V_w$ are all relatively slow-varying (see below for a discussion), Theorem~\ref{thm:speedctrl} says that $\omega_r$ goes to $\omega_{rd}$. Thus, after a short transient,
\begin{align}
\omega_r \approx \omega_{rd}. \label{eqn:Vwr}
\end{align}
Combining \eqref{eqn:Vf1}--\eqref{eqn:Vwr}, \eqref{eqn:Vf} is obtained as claimed.

Now observe that the first three variables $(\omega_{rd}, \theta, \beta)$ in \eqref{eqn:Vf} are yet to be determined, while the last three variables $(V_w, P_d, Q_d)$ are exogenous but known. Therefore, for each given $(V_w, P_d, Q_d)$, $(\omega_{rd}, \theta, \beta)$ can be chosen correspondingly in order to minimize $V$. This defines a mapping from $(V_w, P_d, Q_d)$ to $(\omega_{rd}, \theta, \beta)$, i.e.,
\begin{align}
(\omega_{rd}, \theta, \beta)&=F(V_w, P_d, Q_d)\nonumber\\
&\triangleq \operatorname{arg\,min}_{(x_1, x_2, x_3)}f(x_1,\!x_2,\!x_3,\!V_w,\!P_d,\!Q_d). \label{eqn:F}
\end{align}

In principle, the mapping $F$ in \eqref{eqn:F} may be constructed either {\em analytically}, by setting the gradient of $f(\cdot)$ to zero and solving for the minimizer $(\omega_{rd}, \theta, \beta)$ in terms of $(V_w, P_d, Q_d)$, or {\em numerically}, by means of a three-dimensional lookup table. Unfortunately, the former is difficult to carry out, since $f$, being composed of several nonlinear transformations \eqref{eqn:Vf1}--\eqref{eqn:Vwr}, has a rather complex expression. On the other hand, the latter is costly to generate and can easily become obsolete due to variations in system parameters. More important, selecting $(\omega_{rd}, \theta, \beta)$ as a {\em static} function of $(V_w, P_d, Q_d)$ as in \eqref{eqn:F} may lead to steep jumps in $(\omega_{rd}, \theta, \beta)$ because $V_w$ is ever-changing and may change dramatically, and both $P_d$ and $Q_d$ from the wind farm operator may experience step changes. Such steep jumps are undesirable because large fluctuations in $\omega_{rd}$ may prevent $\omega_r$ from tracking it, while discontinuous changes in $\beta$ may be mechanically impossible to realize, cause intolerable vibrations, and substantially cut short the lifetime of the turbine blades.

To alleviate the aforementioned deficiencies of selecting $(\omega_{rd}, \theta, \beta)$ according to \eqref{eqn:F}, a gradient-based approach is considered for updating $(\omega_{rd}, \theta, \beta)$:
\begin{align}
\dot{\omega}_{rd}&=-\epsilon_1\frac{\partial f}{\partial \omega_{rd}},\label{eqn:omegarddot}\\
\dot{\theta}&=-\epsilon_2\frac{\partial f}{\partial \theta},\label{eqn:thetadot}\\
\dot{\beta}&=-\epsilon_3\frac{\partial f}{\partial \beta},\label{eqn:betadot}
\end{align}
where $\epsilon_1, \epsilon_2, \epsilon_3>0$ are design parameters, which are meant to be relatively small, especially $\epsilon_1$ and $\epsilon_3$, in order to avoid steep changes in $\omega_{rd}$ and $\beta$. The partial derivatives $\frac{\partial f}{\partial \omega_{rd}}$, $\frac{\partial f}{\partial \theta}$, and $\frac{\partial f}{\partial \beta}$ in \eqref{eqn:omegarddot}--\eqref{eqn:betadot} can be calculated in a straightforward manner using \eqref{eqn:Vf1}--\eqref{eqn:Vwr}, but are omitted from this paper due to space limitations. These partial derivatives are practically implementable since, like $f$, they depend on $\omega_{rd}$, $\theta$, $\beta$, $V_w$, $P_d$, and $Q_d$, all of which are known. With this gradient-based approach, $(\omega_{rd}, \theta, \beta)$ is guaranteed to asymptotically converge to a local minimum when $(V_w, P_d, Q_d)$ is constant, and track a local minimum when $(V_w, P_d, Q_d)$ varies.

To help the readers better understand the proposed nonlinear controller depicted in Figure~\ref{fig:wtc} and described in this section, the internal structure of this controller is revealed in Figure~\ref{fig:ctrller}. Observe that each arrow in this figure represents a signal, whereas each tiny box represents equations relating the signals.

\begin{figure}[tb]
\centering\begin{texdraw}
\drawdim in \setunitscale 1
\move(0 0.1) \lvec(0.4 0.1) \arrowheadtype t:F \arrowheadsize l:0.08 w:0.04 \avec(0.6 0.3)
\move(0 0.5) \lvec(0.4 0.5) \avec(0.6 0.3) \lvec(0.8 0.3) \avec(1.0 0.3) \lvec(1.28 0.3) \move(1.3 0.3) \larc r:0.02 sd:0 ed:180 \move(1.32 0.3) \lvec(1.9 0.3) \avec(1.9 0.9)
\move(0.8 0.3) \avec(1.0 0.1) \lvec(1.3 0.1) \avec(1.3 0.9)
\move(0.8 0.3) \avec(1.0 0.5) \lvec(1.28 0.5) \move(1.3 0.5) \larc r:0.02 sd:0 ed:180 \move(1.32 0.5) \lvec(1.6 0.5) \avec(1.6 0.9)
\move(0 0.7) \lvec(0.4 0.7) \avec(0.6 0.9) \move(0.4 0.7) \lvec(1.1 0.7) \avec(1.3 0.9)
\move(0 1.1) \lvec(0.4 1.1) \avec(0.6 0.9) \avec(1.3 0.9) \avec(1.6 0.9) \avec(1.9 0.9) \avec(2.4 0.9) \avec(3.15 0.9)
\move(1.3 0.1) \avec(3.15 0.1)
\move(0.4 1.1) \lvec(2.2 1.1) \avec(2.4 0.9) \move(1.4 1.1) \avec(1.6 0.9) \move(1.1 1.1) \avec(1.3 0.9)
\move(0 1.3) \avec(1.3 1.3) \lvec(2.4 1.3) \avec(2.4 0.9)
\linewd 0.02 \move(2.875 0.01) \lvec(2.875 1.45) \lvec(0.35 1.45) \lvec(0.35 0.01) \lvec(2.875 0.01)
\move(0.58 0.28) \linewd 0.01 \lvec(0.62 0.28) \lvec(0.62 0.32) \lvec(0.58 0.32) \lvec(0.58 0.28) \lfill f:1
\move(0.58 0.88) \linewd 0.01 \lvec(0.62 0.88) \lvec(0.62 0.92) \lvec(0.58 0.92) \lvec(0.58 0.88) \lfill f:1
\move(0.98 0.08) \linewd 0.01 \lvec(1.02 0.08) \lvec(1.02 0.12) \lvec(0.98 0.12) \lvec(0.98 0.08) \lfill f:1
\move(0.98 0.28) \linewd 0.01 \lvec(1.02 0.28) \lvec(1.02 0.32) \lvec(0.98 0.32) \lvec(0.98 0.28) \lfill f:1
\move(0.98 0.48) \linewd 0.01 \lvec(1.02 0.48) \lvec(1.02 0.52) \lvec(0.98 0.52) \lvec(0.98 0.48) \lfill f:1
\move(1.28 0.88) \linewd 0.01 \lvec(1.32 0.88) \lvec(1.32 0.92) \lvec(1.28 0.92) \lvec(1.28 0.88) \lfill f:1
\move(1.58 0.88) \linewd 0.01 \lvec(1.62 0.88) \lvec(1.62 0.92) \lvec(1.58 0.92) \lvec(1.58 0.88) \lfill f:1
\move(1.88 0.88) \linewd 0.01 \lvec(1.92 0.88) \lvec(1.92 0.92) \lvec(1.88 0.92) \lvec(1.88 0.88) \lfill f:1
\move(2.38 0.88) \linewd 0.01 \lvec(2.42 0.88) \lvec(2.42 0.92) \lvec(2.38 0.92) \lvec(2.38 0.88) \lfill f:1
\move(1.28 1.28) \linewd 0.01 \lvec(1.32 1.28) \lvec(1.32 1.32) \lvec(1.28 1.32) \lvec(1.28 1.28) \lfill f:1
\textref h:R v:B \htext(0.3 0.11){$P, Q$}
\textref h:R v:B \htext(0.325 0.51){$P_d, Q_d$}
\textref h:R v:B \htext(0.25 0.71){$V_w$} \htext(0.25 1.11){$\omega_r$}
\textref h:L v:B \htext(0.7 0.31){$V$} \htext(1.05 0.91){$\lambda$} \htext(1.35 0.91){$T_m$} \htext(1.7 0.91){$r$} \htext(1.95 0.91){$u_1, u_2$} \htext(2.9 0.91){$v_{dr}, v_{qr}$}
\htext(1.6 0.11){$\beta$} \htext(2.95 0.11){$\beta$} \htext(1.5 0.31){$\theta$} \htext(1.35 0.51){$\omega_{rd}$}
\textref h:R v:B \htext(0.25 1.31){$i$}
\textref h:L v:B \htext(1.85 1.31){$x$}
\htext(0.6 0.15){\eqref{eqn:objfun}} \htext(0.6 0.75){\eqref{eqn:lambda}} \htext(1.1 1.15){\eqref{eqn:ids}--\eqref{eqn:iqr}}
\htext(1.3 0.75){\eqref{eqn:Tm}} \htext(1.6 0.75){\eqref{eqn:r}} \htext(1.9 0.75){\eqref{eqn:ztheta},\eqref{eqn:zfunu}}
\htext(2.4 0.75){\eqref{eqn:vdr1},\eqref{eqn:vqr1}}
\htext(1.0 0.1){\eqref{eqn:betadot}} \htext(1.0 0.3){\eqref{eqn:thetadot}} \htext(1.0 0.5){\eqref{eqn:omegarddot}}
\end{texdraw}
\caption{Internal structure of the proposed nonlinear controller.}
\label{fig:ctrller}
\end{figure}

\section{Simulation Studies}\label{sec:simstud}

\begin{figure*}[tb]
\centering\includegraphics[width=0.95\textwidth]{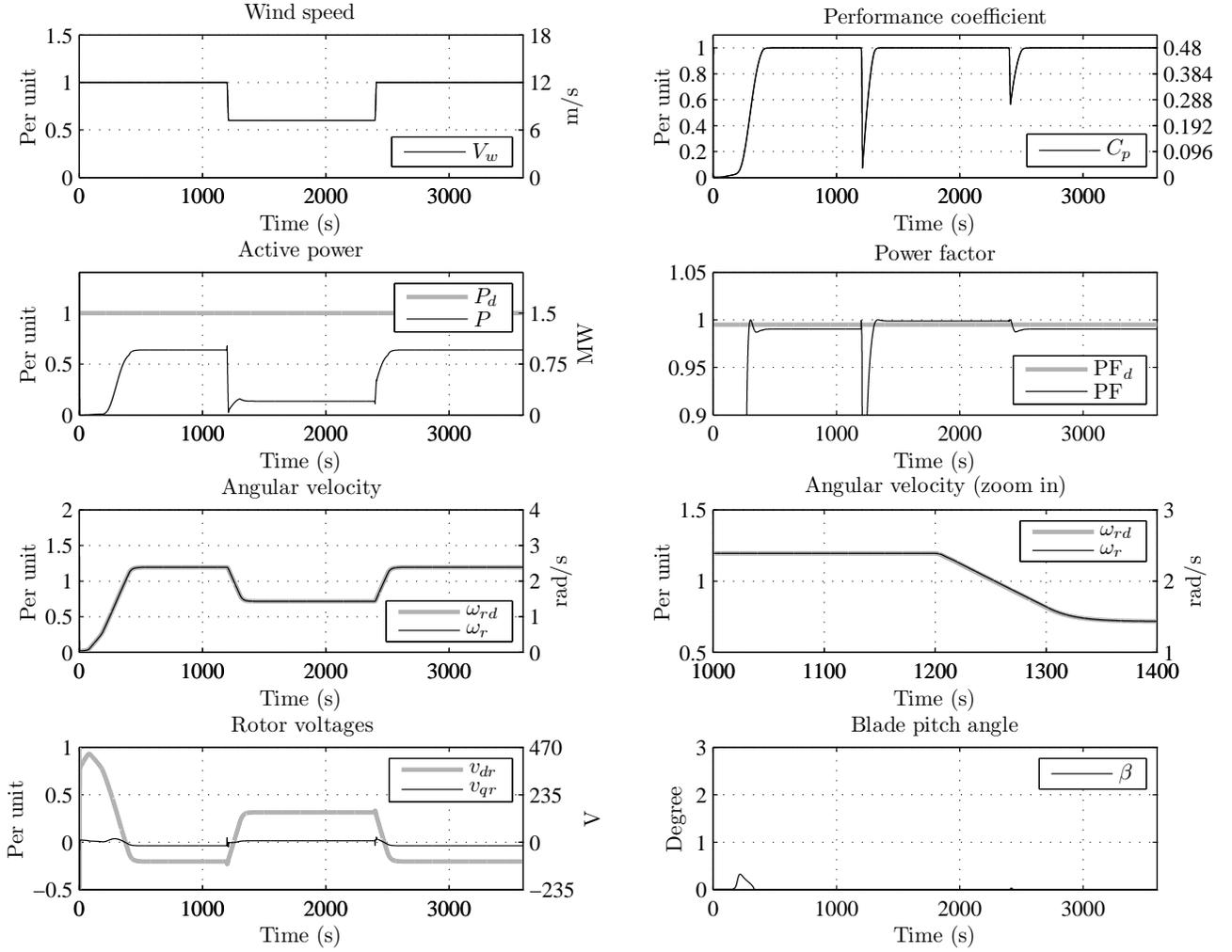}
\caption{Scenario 1 illustrating the maximum power tracking (MPT) mode.}
\label{fig:res1}
\end{figure*}

To demonstrate the effectiveness of the controller presented above, MATLAB simulations have been carried out, in which the controller is applied to a $1.5\operatorname{MW}$ GE turbine with $575\operatorname{V}$ base voltage and $60\operatorname{Hz}$ base frequency. To describe settings and results of the simulations, both the per-unit system and the physical unit system will be used, given that they are popular in the literature.

The simulation settings are as follows: All values of the wind turbine parameters are taken from the Wind Turbine Block of the Distributed Resources Library in MATLAB/Simulink R2007a. Specifically, the values are: $\omega_s(\operatorname{pu})=1$, $R_s(\operatorname{pu})=0.00706$, $R_r(\operatorname{pu})=0.005$, $L_s(\operatorname{pu})=3.071$, $L_r(\operatorname{pu})=3.056$, $L_m(\operatorname{pu})=2.9$, $J(\operatorname{pu})=10.08$, $C_f(\operatorname{pu})=0.01$, $\lambda_{\text{nom}}=8.1$, and $C_{p\_\text{nom}}=0.48$. In addition, the base wind speed is $V_{w\_\text{base}}=12\operatorname{m/s}$. In MATLAB/Simulink R2007a, the mechanical power captured is given by
\begin{align}
P_m(\operatorname{pu})=\frac{P_{\text{wind}\_\text{base}}P_\text{nom}}{P_{\text{elec}\_\text{base}}}\,C_p(\operatorname{pu})\,V_w(\operatorname{pu})^3,\label{eqn:Pmpu}
\end{align}
where $P_{\text{wind}\_\text{base}}=1.5\operatorname{MW}$, $P_{\text{nom}}=0.73$, and $P_{\text{elec}\_\text{base}}=1.5\times10^6/0.9\operatorname{VA}$. From \eqref{eqn:Pmpu}, it can be seen that at the base wind speed $V_w(\operatorname{pu})=1$, $P_m(\operatorname{pu})$ is capped at $0.657$. Furthermore, without loss of generality, the constant stator voltages are set to $v_{ds}(\operatorname{pu})=1$ and $v_{qs}(\operatorname{pu})=0$.

For the proposed controller, we let the desired poles of the electrical dynamics \eqref{eqn:dotx} be located at $-15$, $-5$, and $-10\pm5j$, so that the corresponding state feedback gain matrix $K$, calculated using MATLAB's {\tt place()} function, is
\begin{align*}
K=
\begin{bmatrix}
5135.9 & 259.2 & 20.3 & 1.9\\
-2676.7 & 4289.9 & -1.3 & 19.7
\end{bmatrix}.
\end{align*}
In addition, we let $w_p=10$, $w_q=1$, and $w_{pq}=0$, implying that we penalize the difference between $P$ and $P_d$ much more than we do $Q$ and $Q_d$. Finally, we choose the rest of the controller parameters as follows: $\alpha=10$, $\epsilon_1=4\times10^{-3}$, $\epsilon_2=1\times10^{-4}$, and $\epsilon_3=2$.

Based on the above wind turbine and controller parameters, simulations have been carried out for four different scenarios. Description of each scenario, along with the simulation result, is given below:

\begin{figure*}[tb]
\centering\includegraphics[width=0.95\textwidth]{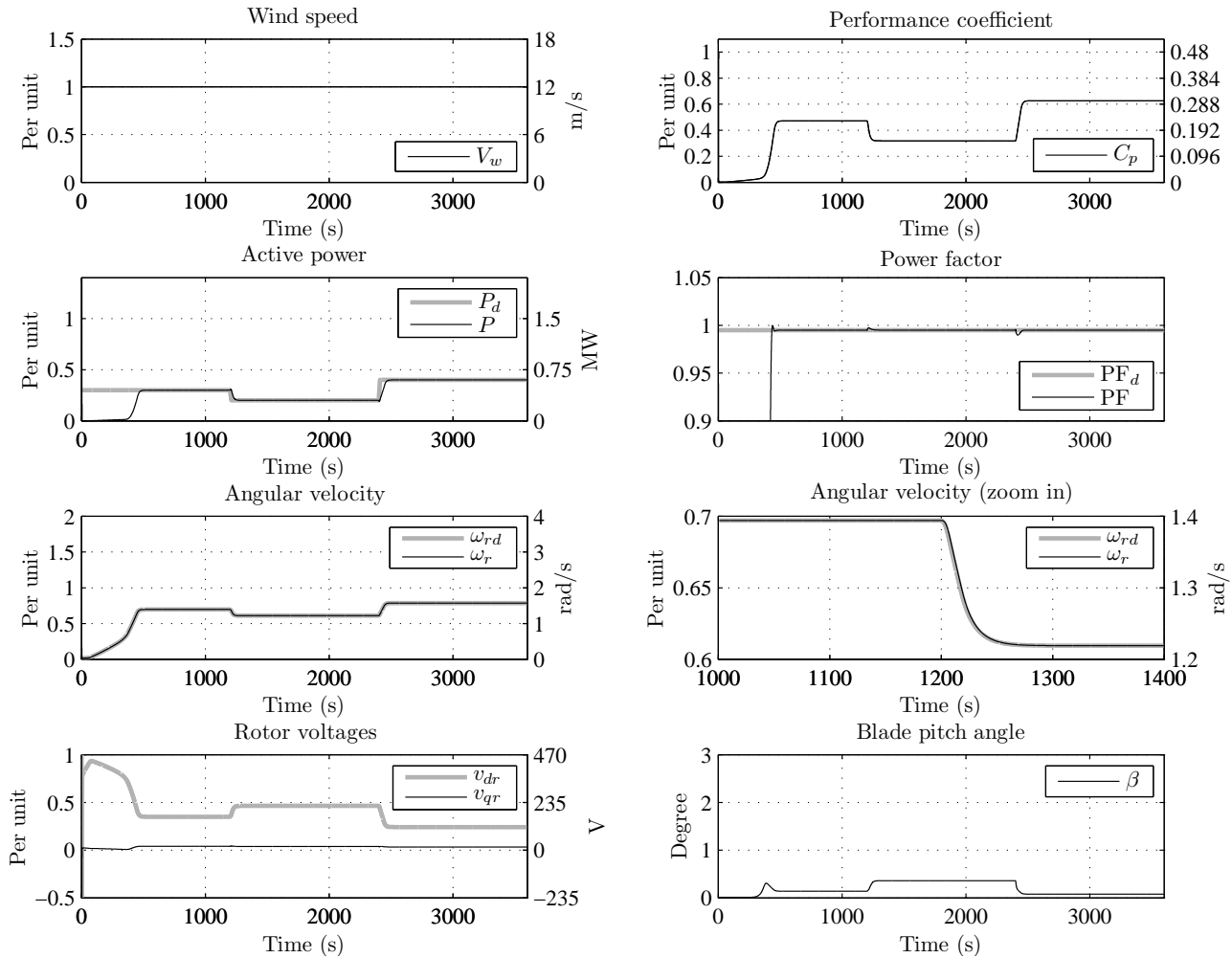}
\caption{Scenario 2 illustrating the power regulation (PR) mode.}
\label{fig:res2}
\end{figure*}

{\bf Scenario 1}: {\em Maximum power tracking (MPT) mode}.
In this scenario, we simulate the situation where the wind speed $V_w$ experiences step changes between $12\operatorname{m/s}$ and $7.2\operatorname{m/s}$, while the desired powers $P_d$ and $Q_d$ are kept constant at $1.5\operatorname{MW}$ and $0.15\operatorname{MW}$, so that the desired power factor is $\text{PF}_d=0.995$. Since $P_m$ cannot exceed $0.657\times1.5\operatorname{MW}$ at the base wind speed $V_{w\_\text{base}}=12\operatorname{m/s}$, the wind turbine is expected to operate in the MPT mode. Figure~\ref{fig:res1} shows the simulation result for this scenario, where the key signals are plotted as functions of time in both the per-unit and physical unit systems wherever applicable. Observe that, after a short transient, the wind turbine converts as much wind energy to electric energy as it possibly could, as indicated by $C_p$ approaching its maximum value of $0.48$ in subplot 2 (which translates into $P$ approaching its maximum possible value in subplot 3). Also observe that, when $V_w$ goes from $12\operatorname{m/s}$ to $7.2\operatorname{m/s}$ and from $7.2\operatorname{m/s}$ back to $12\operatorname{m/s}$, $C_p$ drops sharply but quickly returns to its maximum value. Note from subplot 4 that, regardless of $V_w$, the power factor PF is maintained near the desired level of $0.995$. Moreover, note from subplots 5 and 6 that the angular velocity $\omega_r$ tracks the desired time-varying reference $\omega_{rd}$ closely (subplot 6 is a zoom-in version of subplot 5). Finally, the control inputs $v_{dr}$, $v_{qr}$, and $\beta$ are shown in subplots 7 and 8, respectively. Note that, to maximize $C_p$, $\beta$ is kept at its minimum value $\beta_{\min}=0\operatorname{deg}$.

{\bf Scenario 2}: {\em Power regulation (PR) mode}.
In this scenario, we simulate the situation where $V_w$ is kept constant at the base value of $12\operatorname{m/s}$, while $P_d$ experiences step changes from $0.45\operatorname{MW}$ to $0.3\operatorname{MW}$ and then to $0.6\operatorname{MW}$, and $Q_d$ is such that $\text{PF}_d=0.995$. Since $P_d$ is always less than $0.657\times1.5\operatorname{MW}$ at the base wind speed of $12\operatorname{m/s}$, the wind turbine is expected to operate in the PR mode with different setpoints $P_d$. Figure~\ref{fig:res2} shows the simulation result for this scenario. Observe from subplot 2 that $C_p$ is less than its maximum value of $0.48$. This suggests that the wind turbine attempts to capture less power than what it possibly could from wind, since $P_d$ is relatively small. Indeed, as can be seen from subplots 3 and 4, the turbine produces just enough active and reactive powers, making $P$ track $P_d$ closely while maintaining PF at $\text{PF}_d$. Also observe from subplots 5 and 6 that $\omega_r$ closely follows $\omega_{rd}$, as desired. Finally, note from subplot 8 that $\beta$ increases slightly in order to capture less power between 1200s and 2400s, when $P_d$ is smallest.

\begin{figure*}[tb]
\centering\includegraphics[width=0.95\textwidth]{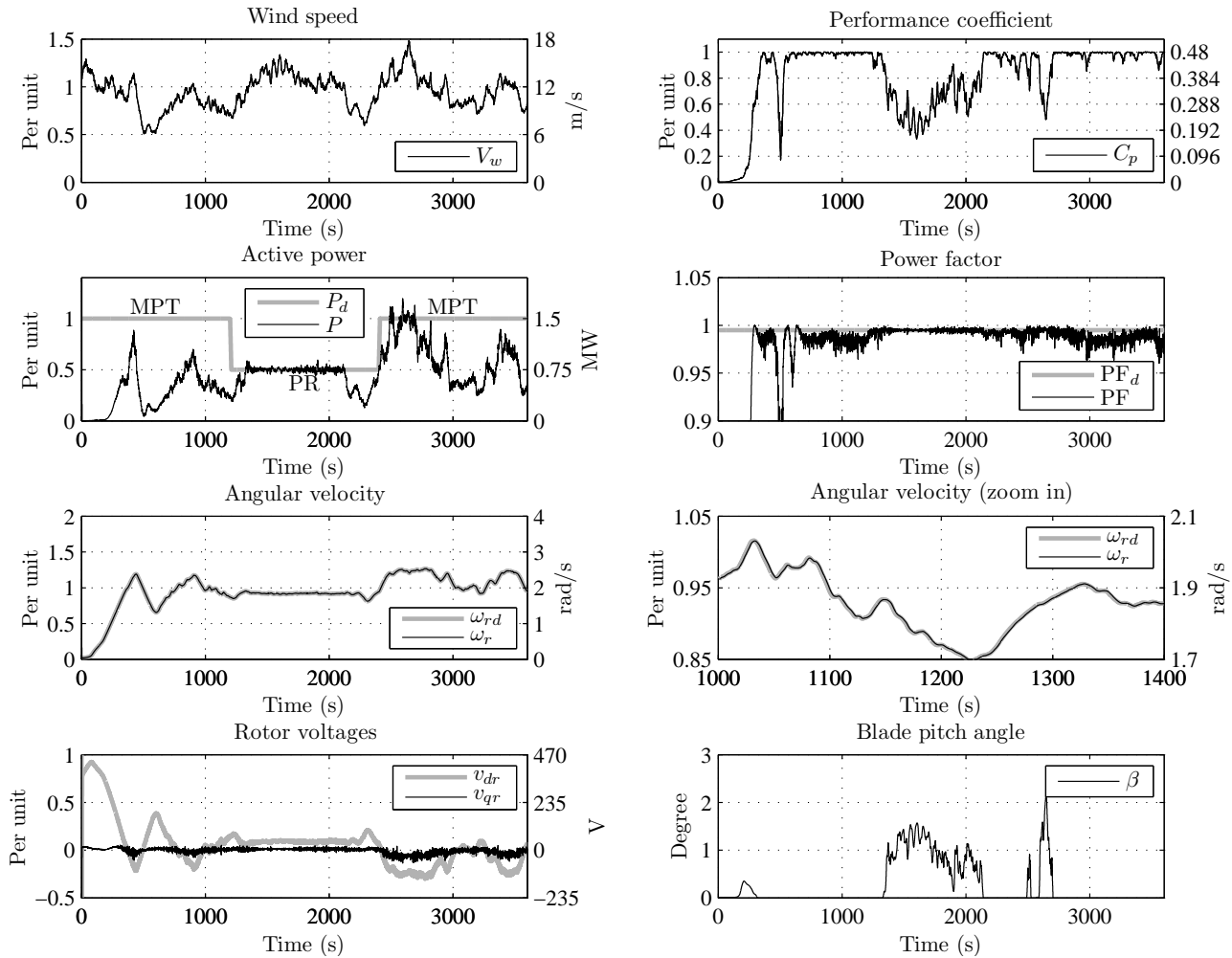}
\caption{Scenario 3 illustrating the seamless switching between the MPT and PR modes under an actual wind profile from a wind farm located in northwest Oklahoma.}
\label{fig:res3}
\end{figure*}

{\bf Scenario 3}: {\em Seamless switching between the MPT and PR modes}.
In this scenario, we simulate the situation where $P_d$ experiences large step changes between $1.5\operatorname{MW}$ and $0.75\operatorname{MW}$, $Q_d$ again is such that $\text{PF}_d$ is $0.995$, and an actual wind profile from a wind farm located in northwest Oklahoma is used to define $V_w$. The actual wind profile consists of 145 samples, taken at the rate of one sample per 10 minutes, over a 24-hour period. In order to use this wind profile in a 1-hour simulation (as in Scenarios~1 and~2), we compress the time scale, assuming that the samples were taken over a 1-hour period. Note that compressing the time scale in this way makes the problem more challenging because the wind speed varies faster than it actually does. Figure~\ref{fig:res3} shows the simulation result for this scenario, with subplot 1 displaying the wind profile. Observe from subplots 2 and 3 that, for the first 1200 seconds during which $P_d$ is $1.5\operatorname{MW}$, the turbine operates in the MPT mode, grabbing as much wind energy as it possibly could, by driving $C_p$ to $0.48$ and maximizing $P$. At time 1200s when $P_d$ abruptly drops from $1.5\operatorname{MW}$ to $0.75\operatorname{MW}$, the turbine seamlessly switches from the MPT mode to the PR mode, quickly reducing $C_p$, accurately regulating $P$ around $P_d$, and effectively rejecting the ``disturbance'' $V_w$. Note that between 2100s and 2400s, the wind is not strong enough to sustain the PR mode. As a result, the MPT mode resumes seamlessly, as indicated by $C_p$ returning immediately to its maximum value of $0.48$. Finally, at time 2400s when $P_d$ goes from $0.75\operatorname{MW}$ back to $1.5\operatorname{MW}$, the turbine keeps working in the MPT mode, continuing to maximize both $C_p$ and $P$. Notice from subplots 4--6 that, over the course of the simulation, both PF and $\omega_r$ are maintained at $\text{PF}_d$ and $\omega_{rd}$, respectively, despite the random wind fluctuations. Also notice from subplot 8 that $\beta$ increases somewhat during the PR mode in order to help capture less power.

\begin{figure}[tb]
\centering\includegraphics[width=0.95\linewidth]{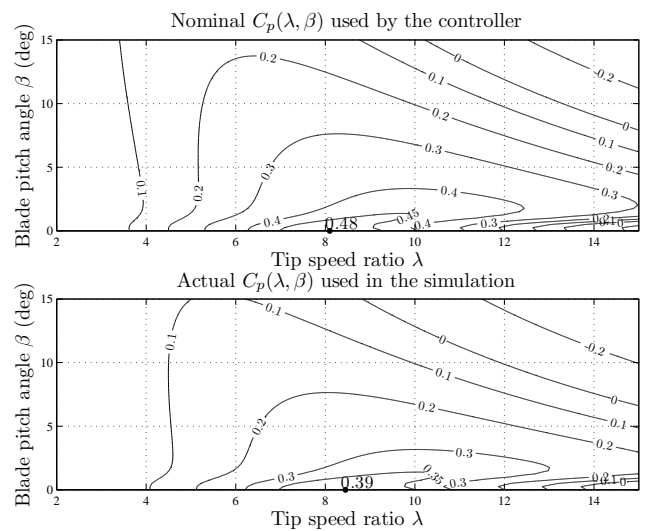}
\caption{Contour plots of the nominal and actual $C_p(\lambda,\beta)$ for Scenario 4.}
\label{fig:cp}
\end{figure}

\begin{figure*}[tb]
\centering\includegraphics[width=0.95\textwidth]{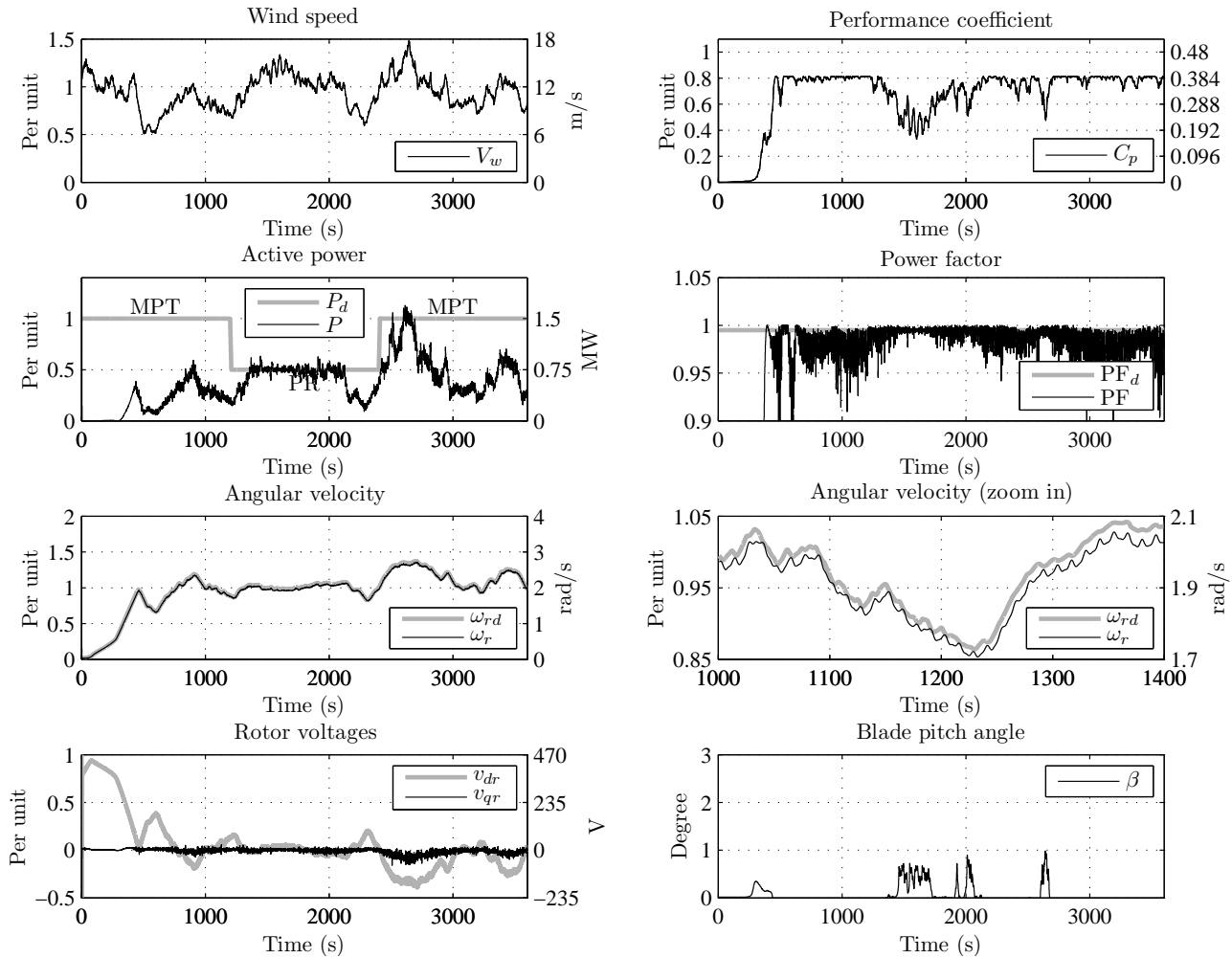}
\caption{Scenario 4 illustrating the robustness of the proposed controller to modeling errors in $C_f$ and $C_p$ and noisy measurements in $V_w$.}
\label{fig:res4}
\end{figure*}

{\bf Scenario 4}: {\em Robustness of the proposed controller}.
In this scenario, we simulate the exact same situation as that of Scenario 3 (i.e., with the same $V_w$, $P_d$, and $Q_d$) but with modeling errors and measurement noise. That is, we allow for modeling errors in the friction coefficient $C_f$ and the performance coefficient $C_p$ (due, for example, to changing weather conditions, blade erosions, and aging) as well as measurement noise in the wind speed $V_w$ (since $V_w$ is usually measured by an anemometer located on the nacelle behind the blades of a wind turbine). Specifically, we assume that the nominal $C_f$ used by the controller is $0.01(\operatorname{pu})$, whereas the actual $C_f$ used in the simulation is $0.012(\operatorname{pu})$, so that $C_f$ has a 20\% modeling error. Moreover, we assume that the nominal $C_p$ used by the controller is given by \eqref{eqn:cp} and \eqref{eqn:lambdai} with $c_1=0.5176$, $c_2=116$, $c_3=0.4$, $c_4=5$, $c_5=21$, and $c_6=0.0068$, whereas the actual $C_p$ used in the simulation is also given by \eqref{eqn:cp} and \eqref{eqn:lambdai} but with $c_1=0.45$, $c_2=115$, $c_3=0.5$, $c_4=4.5$, $c_5=22$, and $c_6=0.003$.
Figure~\ref{fig:cp} displays the contour plots of the nominal and actual $C_p(\lambda,\beta)$ for $\lambda\in[2,15]$ and $\beta\in[0,15]$, showing that $C_p$ has noticeable modeling errors. In particular, the nominal $C_p$ attains its maximum of $0.48$ at $(\lambda, \beta)=(8.1,0)$, whereas the actual $C_p$ attains its maximum of $0.39$ at $(\lambda, \beta)=(8.45,0)$. Finally, we assume that the measured $V_w$ used by the controller, denoted as $V_{w\_\text{meas}}$, is related to the actual $V_w$ used in the simulation via
\begin{align*}
V_{w\_\text{meas}}(t)=V_w(t)+0.5+0.5\sin(0.5t)+0.25\cos(t),
\end{align*}
where the second term on the right-hand side represents a constant measurement bias, while the third and fourth represent measurement noises with different amplitudes and frequencies. Figure~\ref{fig:res4} shows the simulation result for this scenario. Comparing this figure with Figure~\ref{fig:res3}, the following observations can be made: first, $C_p$ in Figure~\ref{fig:res4} attains its maximum value of $0.39$ in the MPT mode, as opposed to the $0.48$ attained by $C_p$ in Figure~\ref{fig:res3}. Second, PF in Figure~\ref{fig:res4} has a larger fluctuation compared to PF in Figure~\ref{fig:res3}, but nonetheless is maintained around $\text{PF}_d$. Third, $\omega_r$ in Figure~\ref{fig:res4} does not track $\omega_{rd}$ as closely as $\omega_r$ in Figure~\ref{fig:res3} does. Nevertheless, despite the wind fluctuations, modeling errors, and noisy measurements, the controller performs reasonably well, as evident by how close $C_p$ is to its maximum value of $0.39$ in the MPT mode, how close $P$ is to $P_d$ in the PR mode, and how close PF is to $\text{PF}_d$ throughout the simulation. Therefore, the controller is fairly robust.

As it follows from Figures~\ref{fig:res1}--\ref{fig:res4} and the above discussions, the proposed controller exhibits excellent performance. Specifically, the controller works well in both the MPT mode under step changes in the wind speed (Scenario 1) and the PR mode under step changes in the power commands (Scenario 2). In addition, it is capable of seamlessly switching between the two modes in the presence of changing power commands and a realistic, fluctuating wind profile (Scenario 3). Finally, the controller is robust to small modeling errors and noisy measurements commonly encountered in practice (Scenario 4).

\section{Conclusion}\label{sec:concl}

In this paper, we have developed a feedback/feedforward nonlinear controller, which accounts for the nonlinearities of variable-speed wind turbines with doubly fed induction generators, and bypasses the need for approximate linearization. Its development is based on applying a mixture of linear and nonlinear control design techniques on three time scales, including feedback linearization, pole placement, and gradient-based minimization of a Lyapunov-like potential function. Simulation results have shown that the proposed scheme not only effectively controls the active and reactive powers in both the MPT and PR modes, it also ensures seamless switching between the two modes. Therefore, the proposed controller may be recommended as a candidate for future wind turbine control.

\bibliographystyle{IEEEtran}
\bibliography{paper}

\begin{IEEEbiographynophoto}{Choon Yik Tang} (S'97--M'04) received the B.S. and M.S. degrees in mechanical engineering from Oklahoma State University, Stillwater, in 1996 and 1997, respectively, and the Ph.D. degree in electrical engineering from the University of Michigan, Ann Arbor, in 2003. From 2003 to 2004, he was a Postdoctoral Research Fellow in the Department of Electrical Engineering and Computer Science at the University of Michigan. From 2004 to 2006, he was a Research Scientist at Honeywell Labs, Minneapolis. Since 2006, he has been an Assistant Professor in the School of Electrical and Computer Engineering at the University of Oklahoma, Norman. His current research interests include systems and control theory, distributed algorithms for computation and optimization over networks, and control and operation of wind farms.
\end{IEEEbiographynophoto}

\begin{IEEEbiographynophoto}{Yi Guo} (S'08) received the B.S. degree from Tianjin Polytechnic University, Tianjin, China, and the M.S. degree from Tianjin University, Tianjin, China, in 2002 and 2005, respectively. He is currently working toward his Ph.D. degree in the School of Electrical and Computer Engineering at the University of Oklahoma, Norman. His current research interests include control theory and applications, with an emphasis on control of wind turbines and wind farms.
\end{IEEEbiographynophoto}

\begin{IEEEbiographynophoto}{John N. Jiang} (SM'07) is an Assistant Professor in the Power System Group in the School of Electrical and Computer Engineering at the University of Oklahoma, Norman. He holds M.S. and Ph.D. degrees from the University of Texas at Austin. He has been involved in a number of wind energy related projects since 1989 in design, installation of stand-alone wind generation systems, the market impact of wind generation in Texas, and recent large-scale wind farms development in Oklahoma.
\end{IEEEbiographynophoto}

\end{document}